\documentclass[a4paper,10pt,twoside,reqno,intlimits]{amsart}
\usepackage{amsfonts, amsmath, amssymb, mathrsfs, float, graphicx, hyperref}

\def\RR{\mathbb{R}}
\def\conv{\text{conv}}
\def\dK{\partial K}
\def\dL{\partial L}
\def\dQ{\partial Q}
\def\cF{\mathcal{F}}

\newtheorem{theorem}{Theorem}
\newtheorem*{theorem*}{Theorem}
\newtheorem*{lemma}{Lemma}
\newtheorem*{claim}{Claim}
\newtheorem*{remark}{Remark}
\newtheorem*{definition}{Definition}

\title{The (B) conjecture for uniform measures in the~plane}
\author{Amir Livne Bar-on}
\address{Tel Aviv University, Tel Aviv 69978, Israel}
\email{Amir Livne Bar-on <livnebaron@mail.tau.ac.il>}
\thanks{Supported in part by a grant from the European Research Council}


\begin{document}

\begin{abstract}
We prove that for any two centrally-symmetric convex shapes $K,L \subset \RR^2$,
the function $t \mapsto |e^t K \cap L|$ is log-concave.
This extends a result of Cordero-Erausquin, Fradelizi and Maurey in the two dimensional case.
Possible relaxations of the condition of symmetry are discussed.
\end{abstract}

\maketitle

\section{Introduction}
\label{sec:intro}

It was conjectured by Banaszcyk (see Lata{\l}a \cite{latala}) that for any convex set
$K \subset \RR^n$ that is centrally-symmetric (i.e., $K = -K$) and for a centered Gaussian measure $\gamma$,
\begin{equation}
\gamma(s^{1-\lambda} t^\lambda K) \ge \gamma(s K)^{1-\lambda} \gamma(t K)^\lambda
\label{eq:first}
\end{equation}
for any $\lambda \in [0,1]$ and $s,t > 0$.

This conjecture was proven in \cite{bconj}, in the equivalent form that the function
${t \mapsto \gamma(e^t K)}$
is log-concave.
The same paper raises the question whether (\ref{eq:first}) remains valid when $\gamma$
is replaced by other log-concave measures.
The proof of (\ref{eq:first}) for \emph{unconditional} sets and log-concave measures was given in
\cite{bconj} as well:
\begin{theorem}[\cite{bconj}, Proposition 9]
\label{thm:uncond}
Let $K \subset \RR^d$ be a convex set and let $\mu$ be a log-concave measure on $\RR^d$,
and assume that both are invariant under coordinate reflections.
Then $t \mapsto \mu(e^t K)$ is a log-concave function.
\end{theorem}

This paper explores the situation in $\RR^2$.
To distinguish this special case, we call a convex set $K \subset \RR^2$, which is compact and has
a non-empty interior, a \emph{shape}.
The main result is
\begin{theorem}
\label{thm:main}
Let $K,L \subset \RR^2$ be centrally-symmetric convex shapes.
Then $${t \mapsto |e^t K \cap L|}$$ is a log-concave function.
\end{theorem}
Here $|\cdot|$ is the Lebesgue measure, so Theorem~\ref{thm:main} is an analog of (\ref{eq:first})
for \emph{uniform} measures -- with density $d\mu(x) = \mathbf{1}_L(x) dx$.
Note that a uniform measure on a set is log-concave if and only if the set is convex.

The condition of central symmetry in Theorem~\ref{thm:main} can be replaced by dihedral symmetry.
For an integer $n \ge 2$, let $D_n$ be the group of symmetries of $\RR^2$ that is generated by two
reflections, one across the axis $Span\{(1,0)\}$ and the other across the axis
$Span\{(\cos\frac{\pi}{n}, \sin\frac{\pi}{n})\}$.
The dihedral group $D_n$ contains $2n$ transformations.
A $D_n$-symmetric shape $A \subset \RR^2$ is one invariant under the action of $D_n$.

\begin{theorem}
\label{thm:dihedral}
Let $n \ge 2$ be an integer, and let $K,L \subset \RR^2$ be $D_n$-symmetric convex shapes.
Then $t \mapsto |e^t K \cap L|$ is a log-concave function.
\end{theorem}

\subsection*{Examples and open questions}

For what sets and measures is (\ref{eq:first}) valid?

The (B)-conjecture, or (\ref{eq:first}), is not necessarily true for measures and sets with just
one axis of symmetry in $\RR^2$.
An example with a log-concave uniform measure is
\begin{align*}
L & = \conv \left\{ (-5,-2) , (0,3) , (5,-2) \right\} \\
K & = [-6, 6] \times [-3, 1]
\end{align*}
The function $t \mapsto |e^t K \cap L|$ is not log-concave in a neighbourhood of $t = 0$.

Another negative result is for quasi-concave measures.
These are measures with density $d\mu(x) = \varphi(x) dx$ satisfying
$\varphi( (1-\lambda) x + \lambda y ) \ge \max\{\varphi(x), \varphi(y)\}$ for all $0 \le \lambda \le 1$.
If
$$ \mu(A) = |A \cap Q | + |A|, \qquad Q = [-1, 1] \times [-1, 1] $$
then the corresponding function $t \mapsto \mu(e^t Q)$ is not log-concave in a neighbourhood of $t = 1$.

The (B)-conjecture for general centrally-symmetric log-concave measures is not settled yet, even in
two dimensions.
It is also of interest to generalize the method of this paper to higher dimensions.

\subsection*{Notation}

For a convex shape $K \subset \RR^2$, its boundary is denoted by $\dK$.
The support function is denoted $h_K(x) = \sup_{y \in K} \langle x, y \rangle$.
The normal map $\nu_K : \dK \to S^1$ is defined for all smooth points on the boundary,
and $\nu_K(p)$ is the unique direction that satisfies
$\langle\nu_K(p),x\rangle = h_K(x)$.
We denote the unit square by $Q = [-1,1] \times [-1,1] = B^2_{\infty}$.
The Hausdorff distance between sets $A, B \subset \RR^n$ is defined as
$d_H(A,B) = \max \{ \sup_{a \in A} d(a, B), \sup_{b \in B} d(b, A) \}$.
The radial function $\rho_K : \RR \to \RR$ of a convex shape $K \subset \RR^2$ is
$\rho_K(\theta) = \max \{ r \in \RR : (r \cos \theta, r \sin \theta) \in K \}$, with period $2 \pi$.

\section{Main result}
\label{sec:main}

This section proves Theorem~\ref{thm:main}:
\begin{theorem*}
Let $K,L \subset \RR^2$ be centrally-symmetric convex shapes.
Then the function $f_{K,L}(t) = |e^t K \cap L|$ is log-concave.
\end{theorem*}

Obviously, it suffices to show log-concavity around $t=0$.

If we consider the space of centrally-symmetric convex shapes in the plane, equipped with the Hausdorff metric $d_H$,
then the operations $K, L \mapsto K \cap L$ and $K \mapsto |K|$ are continuous.
This means that the correspondence $K, L \mapsto f_{K,L}$ is continuous as well.
Since the condition of log-concavity in the vicinity of a point is a closed condition in the space
$C(\RR)$ of bounded continuous functions,
the class of pairs of centrally-symmetric shapes $K, L \subset \RR^2$ for which $f_{K,L}(t)$ is log-concave
near $t = 0$ is closed w.r.t Hausdorff distance.
Thus in order to prove Theorem~\ref{thm:main} it suffices to prove that $f_{K,L}(t)$ is a log-concave function
near $t = 0$ for a dense set in the space of pairs of centrally-symmetric convex shapes.

As a dense subset, we shall pick the class of transversely-intersecting convex polygons.
This class will be denoted by $\cF$.
The elements of $\cF$ are pairs $(K,L)$ of shapes $K,L \subset \RR^2$ that satisfy:
\begin{itemize}
\item The sets $(K,L)$ are centrally-symmetric convex polygons in $\RR^2$.
\item The intersection $\dK \cap \dL$ is finite.
\item None of the points $x \in \dK \cap \dL$ are vertices of $K$ or of $L$.
That is, there is some $\varepsilon > 0$ such that $B(x,\varepsilon) \cap \dK$ and
$B(x,\varepsilon) \cap \dL$ are line segments.
\item For every $x \in \dK \cap \dL$, $\nu_K(x) \neq \nu_L(x)$.
\end{itemize}

\begin{claim}
The class $\cF$ is dense in the space of pairs of centrally-symmetric convex shapes
(with respect to the Hausdorff metric).
\end{claim}

Hence, in order to prove Theorem~\ref{thm:main}, it is enough to consider polygons with transversal intersection.

\subsection*{Deriving a concrete inequality}
\begin{lemma}
If $(K,L) \in \cF$, then $f_{K,L}(t)$ is twice differentiable in some neighbourhood of $t=0$.
\end{lemma}

\begin{remark}
In this case, log-concavity around $t=0$ amounts to the inequality
\begin{align}
\left. \frac{d^2}{dt^2} \log f(t) \right| _{t=0} & \le 0 \notag \\
f(0) \cdot f''(0) & \le f'(0)^2 . \label{eq:derivs-f}
\end{align}
\end{remark}

\begin{proof}
The area of the intersection is
$$ |a K \cap L| = \int _0^a dr \int _{x \in r \dK \cap L} h_K(\nu_K(\tfrac{x}{r})) d\ell , $$
where $d\ell$ is the length element.

Denote
$$ g_{K,L}(r) = \int _{x \in r \dK \cap L} h_K(\nu_K(\tfrac{x}{r})) d\ell . $$
The transversality of the intersection implies that $g_{K,L}(r)$ is continuous near $r=1$.
Therefore $a \mapsto |a K \cap L|$ is continuously differentiable near $a=1$.

The contour $r \dK \cap L$ is a finite union of segments in $\RR^2$.
Transversality implies that the number of connected components does not change with $r$ in a small
neighbourhood of $r=1$.
The beginning and end points of each component are smooth functions of $r$, also in some neighbourhood of $r=1$.
Therefore $g_{K,L}(r)$ is differentiable as claimed.
\end{proof}

Note that in such a neighbourhood of $r=1$, the function $g_{K,L}(r)$ only depends on the parts of $K$ and $L$
that are close to $\dK \cap L$, and is in fact a sum of contributions from each of the connected components.

Writing (\ref{eq:derivs-f}) in terms of $g(r)$, we get the following condition:

\begin{definition}
For convex shapes $(K, L) \in \cF$, we say that $K$ and $L$ satisfy property $B$, or that $B(K, L)$, if
\begin{align}
|K \cap L| \cdot [g_{K,L}(1) + g_{K,L}'(1)] \le g_{K,L}(1)^2 . \label{eq:derivs}
\end{align}
\end{definition}

The set $\cF$ is open with respect to the Hausdorff metric, and in particular,
if $(K, L) \in \cF$ then $(K, r L) \in \cF$ for every $r$ in some neighbourhood of $r = 1$.
If $B(K, r L)$ holds for every $r$ in such a neighbourhood, then $f_{K, L}(t)$ is log-concave in
some neighbourhood of $t = 0$, as
$$ f_{K, L}(t_0 + t) = e^{2 t_0} f_{K, e^{-t_0} L}(t) . $$

Therefore verifying (\ref{eq:derivs}) for all pairs $(K,L) \in \cF$ will prove Theorem~\ref{thm:main}.

\subsection*{Reduction to parallelograms}
Given two polygons $(K,L) \in \cF$,
the intersection $\dK \cap L$ consists of a finite number of connected components.
Due to central symmetry, they come in opposite pairs.
We denote these components by $S_1, \ldots, S_{2n}$, and ${S_{i+n} = \{ -x : x \in S_i \}}$.

We define a pair of convex shapes $K^{(i)}, L^{(i)}$ for each $1 \le i \le n$ via the following properties.
\begin{itemize}
\item
The shape $K^{(i)}$ is the largest convex set whose boundary contains $S_i \cup S_{i+n}$.
Equivalently, denoting by $x_1,x_2$ the endpoints of $S_i$, and by $x$ the solution of the equations
$$ \begin{cases}
  \langle \nu_K(x_1), x \rangle = h_K(\nu_K(x_1)) \\
  \langle \nu_K(x_2), x \rangle = -h_K(\nu_K(x_2))
\end{cases} $$
$K^{(i)} = \conv \left( S_i \cup S_{i+n} \cup \{x, -x\} \right)$.

\item
The shape $L^{(i)}$ is the parallelogram defined by the four lines
$$ \langle \nu_L(x_1), x \rangle = \pm h_L(\nu_L(x_1)) \quad , \quad \langle \nu_L(x_2), x \rangle = \pm h_L(\nu_L(x_2)) $$
\end{itemize}
See Figure~\ref{fig:extension} for examples.

If $S_i$ is a segment then $K^{(i)}$ described above is an infinite strip,
and if $\nu_L(x_1) = \nu_L(x_2)$ then $L^{(i)}$ is an infinite strip.
We would like to work with compact shapes, thus we apply a procedure to modify $K^{(i)}, L^{(i)}$
to become bounded without changing their significant properties.
Transversality implies that the intersection $K^{(i)} \cap L^{(i)}$ is bounded, even if both sets are strips.
For each $1 \le i \le n$ we pick a centrally-symmetric strip $A \subset \RR^2$ such that $A \cap K^{(i)}$
and $A \cap L^{(i)}$ are both bounded, and which contains $K$ and $L$, and whichever of
$K^{(i)},L^{(i)}$ that is bounded.
From now on we replace $K^{(i)}$ and $L^{(i)}$ by their intersection with $A$.

\begin{figure}[h!]
\includegraphics[width=\textwidth]{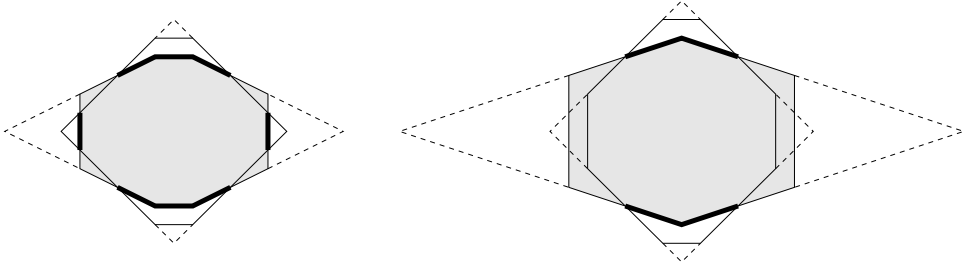}
\caption{Two examples of the extension $K, L \implies K^{(1)}, L^{(1)}$.
\small The shaded shape in each diagram is $K$ and the white shape with a solid boundary line is the corresponding $L$.}
\label{fig:extension}
\end{figure}

\begin{remark}
Note that the sets grow in the process: $K \subset K^{(i)}$ and $L \subset L^{(i)}$ for all $i = 1 \ldots n$.
They satisfy $\dK^{(i)} \cap L^{(i)} = S_i \cup S_{i+n}$.
Also note that if $K$ is a parallelogram then so are the $K^{(i)}$, for every $i$.
It is trivial to check that $(K^{(i)},L^{(i)}) \in \cF$ when $(K,L) \in \cF$.
\end{remark}

\begin{lemma}
If $B(K^{(i)}, L^{(i)})$ for all $i = 1 \ldots n$, then $B(K,L)$.
\end{lemma}
\begin{proof}
The function $g_{K,L}(r)$ takes non-negative values for $r>0$.
In addition, its value is the sum of contributions from the different connected components of ${r\dK \cap L}$.
From transversality, these components vary continuously around ${r=1}$, hence $g_{K,L}'(1)$ is also a sum of
values coming from the different components.
Therefore we can write
\begin{align*}
|K \cap L| \cdot & [g_{K,L}(1) + g_{K,L}'(1)] = |K \cap L| \cdot \sum _{i=1}^{n} \left[ g_{K^{(i)},L^{(i)}}(1) + g_{K^{(i)},L^{(i)}}'(1) \right] \\
  & \le \sum _{i=1}^{n} |K^{(i)} \cap L^{(i)}| \cdot \left[ g_{K^{(i)},L^{(i)}}(1) + g_{K^{(i)},L^{(i)}}'(1) \right] \\
  & \underset{\text{by} B(K^{(i)},L^{(i)})}{\le} \sum _{i=1}^{n} g_{K^{(i)},L^{(i)}}(1)^2 \le \left( \sum _{i=1}^{n} g_{K^{(i)},L^{(i)}}(1) \right) ^2 = g_{K,L}(1)^2
\qedhere
\end{align*}
\end{proof}

\begin{lemma}
If $B(K, L)$ holds for all pairs of parallelograms $(K,L) \in \cF$, then Theorem~\ref{thm:main} follows.
\end{lemma}
\begin{proof}
Let $(K, L) \in \cF$ be any polygons.
Construct the sequence of pairs $K^{(i)}, L^{(i)}$ from $K,L$.
The shape $L^{(i)}$ is a parallelogram for every $i$.
Then construct the pairs $\left(L^{(i)}\right)^{(j)}, \left(K^{(i)}\right)^{(j)}$ from $L^{(i)},K^{(i)}$, for all $i$.
The shapes $\left(L^{(i)}\right)^{(j)}$ and $\left(K^{(i)}\right)^{(j)}$ will be parallelograms for every $i,j$.
Under our assumption, we have $B\left(\left(L^{(i)}\right)^{(j)}, \left(K^{(i)}\right)^{(j)}\right)$.
From this and the previous lemma, $B(L^{(i)},K^{(i)})$ follows.

The property $B$ is symmetric in the shapes.
That is, $B(S, T) \iff B(T, S)$ for all $(S, T) \in \cF$.
This is since $f_{S,T}$ and $f_{T,S}$ differ by a log-linear factor:
$$ f_{S,T}(t) = |e^t S \cap T| = e^{2t} f_{T,S}(-t) $$

This means that we have $B(K^{(i)},L^{(i)})$ as well.
Applying the previous lemma again gives $B(K, L)$.
\end{proof}

All that remains in order to deduce Theorem~\ref{thm:main} is to analyse the case of centrally-symmetric
parallelograms.

If $K,L$ are parallelograms and $K = T Q$ where $T$ is an invertible linear map and $Q = [-1,1] \times [-1,1]$,
$$ f_{K,L} = \det T \cdot f_{Q,T^{-1}L} . $$

Therefore we can take one of the parallelograms to be a square.
In other words, establishing $B(Q, L)$ where $Q$ is the unit square and
$L$ is a parallelogram, and $(Q, L) \in \cF$, will imply Theorem~\ref{thm:main}.

In fact, we may place additional geometric constraints on the square and the parallelogram.

If neither $Q$ nor $L$ contains a vertex of the other quadrilateral in its interior,
then $\dQ \cap L$ has 4 connected components.
Applying the reduction above to $Q, L$ gives $Q^{(i)},L^{(i)}$ with $i=1,2$,
and the intersection $\dQ^{(i)} \cap L^{(i)}$ has only 2 connected
components, as remarked above.

Since the shapes are convex, if all the vertices of one shape are contained in the other, we have
$Q \subset L$ or $L \subset Q$, and then (\ref{eq:derivs}) holds trivially.
If $L$ contains corners of $Q$ but $Q$ does not contain vertices of $L$, we shall swap them.

These arguments leave two cases to be considered:
\begin{itemize}
\item $Q$ contains 2 vertices of $L$, and $L$ does not contain corners of $Q$. In this case the
intersection $\dQ \cap L$ is contained in two opposite edges of $Q$.
\item $Q$ contains 2 vertices of $L$, and $L$ contains 2 corners of $Q$. In this case the intersection
$\dQ \cap L$ is a subset of the edges around these corners of $Q$.
\end{itemize}

\subsection*{Computation of the special cases}

These cases are defined by 4 real parameters -- the coordinates of the vertices of $L$.
A symbolic expression for $f(t)$ can be derived, and (\ref{eq:derivs}) will be a polynomial inequality in
these parameters.
The geometric conditions given above are also polynomial inequalities in these parameters.
Thus each of the two cases can each be expressed by a universally-quantified formula in
the language of real closed fields.
By Tarski's theorem \cite{tarski}, this first order theory has a decision procedure.
This is implemented in the QEPCAD B computer program \cite{qepcad}.
Relevant computer files, for generation of the symbolic condition and for running the logic solver,
for one of the two cases above, are available at

{\smaller \url{http://www.tau.ac.il/~livnebaron/files/bconj_201311/bconj_corners.mac} }

{\smaller \url{http://www.tau.ac.il/~livnebaron/files/bconj_201311/bconj_qelim.txt} }

A human-readable proof of both cases is included here as well.

\begin{lemma}
If $L$ is a centrally-symmetric parallelogram that satisfies $(Q, L) \in \cF$,
and if $L$ crosses $Q$ only inside the vertical edges of $Q$, then $B(Q, L)$.
\end{lemma}
\begin{proof}
Let $\alpha, \beta, c, d$ be as in Figure~\ref{fig:lemma1}.
\begin{figure}[h!]
\includegraphics[width=0.75\textwidth]{bconj_3.eps}
\caption{}
\label{fig:lemma1}
\end{figure}

The equations for the edges of $L$ are
\begin{align*}
x \cos \alpha + y \sin \alpha & = \pm (c \cos \alpha + d \sin \alpha) \\
x \cos \beta  + y \sin \beta & = \pm (c \cos \beta + d \sin \beta)
\end{align*}
Relevant parameters are computed as follows:
\begin{align*}
\dQ \cap \dL & = \left\{ \pm \big( 1, (c-1) \cot \alpha + d \big) , \pm \big( 1, (c-1) \cot \beta + d \big) \right\} \\
g_{Q,L}(1) & = 2(c-1)(\cot \alpha - \cot \beta) \\
g_{Q,L}'(1) & = -2(\cot \alpha - \cot \beta) \\
g_{Q,L}(1) + g_{Q,L}'(1) & = (2c-4)(\cot \alpha - \cot \beta)
\end{align*}

The area of $L$ is comprised of $Q \cap L$ and of two triangles.
The area of the triangles is $\frac{1}{2}g(1) \cdot (c-1)$ so
$$ |Q \cap L| = |L| - (c-1)^2 (\cot \alpha - \cot \beta) . $$

Note that $0 < \alpha < \frac{\pi}{2} < \beta < \pi$ so $\cot \alpha - \cot \beta$ is a positive quantity,
and that if $c < 2$ the value of $g(1) + g'(1)$ is negative so inequality (\ref{eq:derivs}) is satisfied
immediately.

Assume $c > 2$ from now on.
What we need to prove is
$$ (2c-4)(\cot \alpha - \cot \beta) \cdot \big[ |L| - (c-1)^2 (\cot \alpha - \cot \beta) \big] \le 4(c-1)^2(\cot \alpha - \cot \beta)^2 . $$
Or equivalently
$$ (2c-4)|L| \le (c-1)^2 (\cot \alpha - \cot \beta) \cdot ( 4 + 2c-4 ) , $$
or still
$$ |L| \le \left(1 + \tfrac{2}{c-2}\right) \cdot \tfrac{1}{2}(c-1)g(1) . $$

The amount $\frac{1}{2}(c-1) g(1)$ is the area of the triangles $L \backslash Q$.
By convexity the area of $L$ cannot be larger than that times $\left(\frac{c}{c-1}\right)^2$.
It remains to verify that for $c>2$, $\frac{c^2}{(c-1)^2} < 1 + \frac{2}{c-2}$. This is a
simple exercise in algebra:
$$ \frac{c^2}{(c-1)^2} = 1 + \frac{2c-1}{(c-1)^2}
= 1 + \tfrac{2}{c-2} \cdot \frac{(c-\tfrac{1}{2})(c-2)}{(c-1)^2}
= 1 + \tfrac{2}{c-2} \left[ 1 - \frac{c/2}{(c-1)^2} \right]
\le 1 + \tfrac{2}{c-2} $$
\end{proof}

\begin{lemma}
If $L$ is a centrally-symmetric parallelogram that satisfies $(Q, L) \in \cF$,
and each of $Q, L$ contains two vertices of the other, then $B(Q, L)$.
\end{lemma}
\begin{proof}
Let $a$ and $b$ be as in Figure~\ref{fig:lemma2}, and let $S$ stand for the area $S = |Q \cap L|$.
The numbers $a$ and $b$ are in the range $0 < a, b < 2$, and $\alpha$ and $\beta$ satisfy
$\frac{1}{2} \pi < \alpha < \beta < \pi$.
The area $S$ is in the range $4 - ab < S < 4$.

\begin{figure}[h!]
\includegraphics[width=0.6\textwidth]{bconj_4.eps}
\caption{}
\label{fig:lemma2}
\end{figure}

The quantity $g(1)$ is simply $8 - 2a - 2b$, and $g'(1)$ will soon be shown to be bounded by
$$ g'(1) \le -8 \frac{ S - (4 - ab) } { (4 - S) + \frac{1}{2} (a-b)^2 } . $$

This gives an inequality in the 3 variables $a, b, S$, which will be proved for values in the prescribed ranges.

The length of each dotted line in Figure~\ref{fig:lemma2} is $(a^2 + b^2)^{1/2}$.
Denoting the height of the triangle (the distance between $p$ and the closest dotted line) by $h$,
the area is
$$ S = (4 - ab) + 2 \cdot \frac{1}{2} h \cdot (a^2 + b^2)^{1/2} , $$
so
$$h = \frac{S - (4 - ab)}{(a^2 + b^2)^{1/2}} . $$

The formula for $g'(1)$ in terms of the angles $\alpha, \beta$ is
$$ g'(1) = 4 + 2 \tan \alpha + 2 \cot \beta . $$

Denote $c = \beta - \alpha$.
Holding $c$ fixed, the function
$$ \alpha \mapsto g'(1) = 4 + 2 \tan \alpha + 2 \cot (\alpha + c) $$
is concave and takes the same value for $\alpha$ as for $\frac{3}{2}\pi - c - \alpha$.
Therefore its maximum is attained at $\alpha = \frac{3}{4}\pi - \frac{1}{2}c$.
This gives a bound for $g'(1)$ for a given $c = \beta - \alpha$:
$$ g'(1) \le 4 + 2 \tan \left( \tfrac{3}{4}\pi - \tfrac{1}{2}c \right) + 2 \cot \left( \tfrac{3}{4}\pi + \tfrac{1}{2}c \right) . $$
This bound is stronger for higher values of $c$, since $\tan$ is an increasing function and $\cot$ is
a decreasing function.

The angle between the edges of $L$ meeting at $p$ is $\pi - (\beta - \alpha) = \pi - c$.
When $a$, $b$, and $h$ are kept fixed, the position of $p$ gives a bound for $g'(1)$.
This bound is the weakest when the angle $\pi - c$ is largest.
Simple geometric considerations show that in a family of triangles with the same base and height,
the apex angle is largest when the triangle is isosceles, so we will pursue the case where the triangle
formed by $p$ and the nearest dotted line is isosceles.

The value of $c$ in this case is $c = 2 \tan^{-1} \frac{\frac{1}{2}(a^2 + b^2)^{1/2}}{h}$,
and we get
\begin{align*}
g'(1) & \le 4 + 2 \tan \left( \tfrac{3}{4}\pi - \tfrac{1}{2}\pi + \tan^{-1} \frac{\frac{1}{2}(a^2 + b^2)^{1/2}}{h} \right)
                      + 2 \cot \left( \tfrac{3}{4}\pi + \tfrac{1}{2}\pi - \tan^{-1} \frac{\frac{1}{2}(a^2 + b^2)^{1/2}}{h} \right) \\
& = 4 + 4 \tan \left( \tfrac{1}{4}\pi + \tan^{-1} \frac{1}{2} \frac{a^2 + b^2}{S - (4 - ab)} \right) \\
& = 4 + 4 \cdot \frac{ 1 + \frac{1}{2} \frac{a^2 + b^2}{S - (4 - ab)} }
                                { 1 - \frac{1}{2} \frac{a^2 + b^2}{S - (4 - ab)} }
   = \frac{8}{1 - \frac{1}{2} \frac{a^2 + b^2}{S - (4 - ab)}}
   = -8 \frac{ S - (4 - ab) } { (4 - S) + \frac{1}{2} (a-b)^2 } ,
\end{align*}
which proves the forementioned bound for $g'(1)$.

Therefore, to prove (\ref{eq:derivs}) it is enough to show
$$ S \cdot \left( 8 - 2a - 2b -8 \frac{ S - (4 - ab) } { (4 - S) + \frac{1}{2} (a-b)^2 } \right) \le (8 - 2a - 2b)^2 $$

Rearranging and taking into account that $S < 4$, this is equivalent to
$$ \underbrace{ (8 - 2a - 2b)(8 - 2a - 2b - S)\left( (4 - S) + \tfrac{1}{2} (a-b)^2 \right) + 8S(S - (4 - ab)) }_E \ge 0 $$

When $a$ and $b$ are held fixed, this is a 2\textsuperscript{nd} degree condition on $S$.
Since ${0 < a, b < 2}$, the value and the first two derivatives in the point $S = 4 - ab$ are positive:
\begin{align*}
E |_{S=4-ab} & = (8 - 2a - 2b)(2-a)(2-b)\cdot\tfrac{1}{2}(a^2+b^2) > 0 \\
\left. \frac{\partial E}{\partial S} \right|_{S=4-ab} & = (a+b) \left( (5 - a - b)^2 - 1 \right) + 2(a-b)^2 > 0 \\
\left. \frac{\partial^2 E}{\partial S^2} \right|_{S=4-ab} & = 18(4 - ab) > 0
\end{align*}

This means that the condition stays true for all $S > 4 - ab$, as required.
\end{proof}

\section{Dihedral symmetry}
\label{sec:dihedral}

This section deals with dihedrally symmetric sets.
The group $D_n$ is defined in the introduction.

\begin{theorem*}[\ref{thm:dihedral}]
Let $n \ge 2$ be an integer, and let $K,L \subset \RR^2$ be $D_n$-symmetric convex shapes.
Then $t \mapsto |e^t K \cap L|$ is a log-concave function.
\end{theorem*}

For $n=2$ the group $D_n$ is generated by reflections across the standard axes.
This corresponds to unconditional sets and functions,
and Theorem~\ref{thm:uncond} from \cite{bconj} solves this case.

The proof for $n \ge 3$ is by reduction to the unconditional case.

A \emph{smooth strongly-convex} shape $K \subset \RR^2$ is one whose boundary is
a smooth curve with strictly positive curvature everywhere.
The radial function $\rho_K$ of a smooth strongly-convex shape $K \subset \RR^2$ is a smooth function.
The boundary $\dK$ is the curve
$$ \gamma_K(\theta) = \left( \rho_K(\theta) \cos \theta, \rho_K(\theta) \sin \theta \right) . $$
The convexity of $K$ is reflected in the sign of the curvature of $\gamma_K$.
Positive curvature can be written as a condition on the radial function:
\begin{align}
\rho(\theta)^2 + 2\rho'(\theta)^2 - \rho(\theta) \rho''(\theta) > 0 .
\label{eq:curv}
\end{align}

\begin{proof}[Proof of theorem~\ref{thm:dihedral}]

For any $D_n$-symmetric convex shape $K \subset \RR^2$ there is a series of $D_n$-symmetric convex shapes
whose boundaries are smooth and strongly convex curves, and whose Hausdorff limit is $K$.
By the continuity argument from the previous section, the general case follows from the smooth case.
From here on, $K$ and $L$ are smooth $D_n$-symmetric shapes.

$D_n$-symmetric shapes correspond to radial functions that are even and have period $\frac{2\pi}{n}$.
These shapes are completely determined by their intersection with the sector
$$ G_n = \left\{ (r \cos \theta, r \sin \theta) : r \ge 0 , \theta \in [0,\tfrac{\pi}{n}] \right\} . $$
Given two such shapes $K,L$ the area function is
$$ f_{K,L}(t) = |e^t K \cap L| = 2n f_{K \cap G_n, L \cap G_n}(t) . $$

Let $K \subset \RR^2$ be a $D_n$-symmetric strongly convex shape, and consider the function
$\tilde{\rho}(\theta) = \rho_K(\frac{2}{n} \theta)$.
This is an even function with period $\pi$.
The function $\tilde{\rho}(\theta)$ also satisfies (\ref{eq:curv}):
\begin{align*}
& \tilde{\rho}(\theta)^2 + \tilde{\rho}'(\theta)^2 - \tilde{\rho}(\theta) \tilde{\rho}''(\theta) = \\
& \qquad \tfrac{4}{n^2} \left( \rho_K(\tfrac{2}{n}\theta)^2 + 2\rho_K'(\tfrac{2}{n}\theta)^2 - \rho_K(\tfrac{2}{n}\theta) \rho_K''(\tfrac{2}{n}\theta) \right) + (1 - \tfrac{4}{n^2}) \rho_K(\tfrac{2}{n}\theta)^2 > 0 .
\end{align*}

This means that $\tilde{\rho}(\theta)$ is the radial function of some $D_2$-symmetric (unconditional)
strongly convex shape.
We denote this $w(K)$: the unique shape that satisfies $\rho_{w(K)}(\theta) = \rho_K(\frac{2}{n} \theta)$.

The following function, also named $w$, is defined on $G_n$:
\begin{align*}
w & \left( \begin{matrix} r \cos \theta \\ r \sin \theta \end{matrix} \right)
= \left( \begin{matrix} r \cos \tfrac{n}{2} \theta \\ r \sin \tfrac{n}{2} \theta \end{matrix} \right) .
\qquad \left( \text{for} \ r \ge 0, \ \theta \in [0,\tfrac{\pi}{n}] \right)
\end{align*}

The point function $w$ is an bijection between $G_n$ and $G_2$.
It relates to the shape function $w$ by the formula
$$ \{ w(x) : x \in K \cap G_n \} = w(K) \cap G_2 . $$
The point function $w$ is differentiable inside $G_n$, and has a constant Jacobian determinant $\frac{n}{2}$.

Hence
$$ f_{K,L}(t) = 2n f_{K \cap G_n,L \cap G_n}(t) = 4 f_{w(K) \cap G_2,w(L) \cap G_2}(t) = f_{w(K),w(L)}(t) , $$
and the theorem follows from the result in \cite{bconj}.

\end{proof}

\end{document}